\documentclass[onefignum,onetabnum]{siamonline190516}

\usepackage{amsfonts,amssymb,amsopn}
\usepackage{soul,changepage}

\usepackage{graphicx}
\usepackage{subfigure}

\ifpdf
  \DeclareGraphicsExtensions{.eps,.pdf,.png,.jpg}
\else
  \DeclareGraphicsExtensions{.eps}
\fi

\usepackage{booktabs}

\usepackage{mathrsfs,mathtools,bbm,bm,stmaryrd}
\usepackage{scalerel}
\usepackage{nicefrac}
\usepackage{color}

\usepackage{enumitem}
\setlist[enumerate]{leftmargin=.5in}
\setlist[itemize]{leftmargin=.5in}

\newsiamremark{remark}{Remark}
\newsiamremark{example}{Example}

\headers{Equivariant Networks for Indirect Measurements}{M.~Beckmann, and N.~Heilenkötter}

\title{Equivariant Neural Networks for Indirect Measurements%
\thanks{Submitted to the editors March 15, 2024.\funding{N. Heilenkötter acknowledges funding by the Deutsche Forschungsgemeinschaft (DFG) – Project number 281474342.}}}

\author{Matthias Beckmann\thanks{Center for Industrial Mathematics, University of Bremen, Germany \& Department of Electrical and Electronic Engineering, Imperial College London, UK 
  (\email{research@mbeckmann.de}).}
\and Nick Heilenkötter\thanks{Center for Industrial Mathematics, University of Bremen, Germany 
  (\email{heilenkoetter@uni-bremen.de}).}}

\ifpdf
\hypersetup{
  pdftitle={Equivariant Neural Networks for Indirect Measurements},
  pdfauthor={M. Beckmann, and N. Heilenkötter}
}
\fi

\def\K			{\mathbb K}
\def\R			{\mathbb R}
\def\C			{\mathbb C}
\def\V          {\mathcal{V}}
\def\W          {\mathcal{W}}
\def\X          {\mathcal{X}}
\def\Y          {\mathcal{Y}}
\def\Z          {\mathcal{Z}}
\def\Yv         {\Y_V}
\def\S          {\mathrm{S}_V}
\def\A          {\mathscr{A}}
\def\B          {\mathscr{B}}
\def\P          {\mathscr{P}}
\def\Px         {\mathscr{P}_\X}
\def\Py         {\mathscr{P}_\Y}
\def\Pz         {\mathscr{P}_\Z}
\def\p          {\pi}
\def\px         {{\p_\X}}
\def\py         {{\p_\Y}}
\def\G          {G}

\def\g          {g}
\def\h          {h}
\def\e          {e}
\def\inv        {^{-1}}
\def\invT       {^{-T}}
\def\N          {N}
\def\range      {\mathrm{Rg}}
\def\kernel     {\mathrm{Ker}}
\def\NN         {\mathscr{N}}
\def\pg         {\p_\G}
\def\Diff       {\mathrm{D}}
\def\diff       {\mathrm{d}}
\def\dint       {\int_{\Omega_\X}}
\def\Aff        {\mathrm{Aff}^{+}(2)}
\def\GL         {\mathrm{GL}^+(2)}
\def\SO         {\mathrm{SO}(2)}
\def\SE         {\mathrm{SE}(2)}
\def\T          {\mathrm{T}(2)}

\def\D          {\mathcal{D}}

\def\Radon		{\mathscr{R}}
\def\Sphere		{\mathbb S}

\newcommand{\conv}[2]{\ast_{\scaleto{#1\to#2\mathstrut}{5.5pt}}}

\begin{document}

\maketitle

% REQUIRED
\begin{abstract}
In recent years, deep learning techniques have shown great success in various tasks related to inverse problems, where a target quantity of interest can only be observed through indirect measurements by a forward operator.
Common approaches apply deep neural networks in a post-processing step to the reconstructions obtained by classical reconstruction methods.
However, the latter methods can be computationally expensive and introduce artifacts that are not present in the measured data and, in turn, can deteriorate the performance on the given task.
To overcome these limitations, we propose a class of equivariant neural networks that can be directly applied to the measurements to solve the desired task.
To this end, we build appropriate network structures by developing layers that are equivariant with respect to data transformations induced by well-known symmetries in the domain of the forward operator.
We rigorously analyze the relation between the measurement operator and the resulting group representations and prove a representer theorem that characterizes the class of linear operators that translate between a given pair of group actions.
Based on this theory, we extend the existing concepts of Lie group equivariant deep learning to inverse problems and introduce new representations that result from the involved measurement operations.
This allows us to efficiently solve classification, regression or even reconstruction tasks based on indirect measurements also for very sparse data problems, where a classical reconstruction-based approach may be hard or even impossible.
We illustrate the effectiveness of our approach in numerical experiments and compare with existing methods.
\end{abstract}

% REQUIRED

% REQUIRED

%%%%%%%%%%%%%%%%%%%%%%%%%%%%%%%%%%

\section{Introduction}
Indirect measurements naturally occur in various applications, where a quantity of interest cannot be observed directly but only through effects of certain processes leading to an {\em inverse problem}.
One prominent example is computerized tomography (CT), where the goal consists in imaging the interior of a scanned object by measuring and processing the attenuation
of X-rays travelling through the object under investigation.
Mathematically, the common setting in inverse problems is the observation of a signal~$x$ through indirect measurements $y = \A x$ with forward operator $\A$.
One particular challenge is that inverse problems are typically ill-posed in the sense that no continuous inverse $\A\inv$ exists so that small noise in the measured data $y$ can lead to large errors in the reconstructed signal $x$.

Recently, deep learning approaches based on neural networks have shown remarkable results in various tasks in the context of inverse problems, like classification or regression tasks.
Such methods typically include the solution of the inverse problem.
In contrast to this, the idea presented in this work is to construct data-efficient neural network structures that operate directly on the measured data $y$ and avoid the need for classical reconstruction or back{\-}projection into the original signal domain.
For this purpose, we follow the concept of equivariant neural networks that leverage the symmetries that may be present in the measurements $y$.

This might become useful especially in settings where only very sparse measurement data is available and, as a result, good reconstructions without additional data-knowledge are infeasible.
A motivating example for our work is the thickness measurement of tubes using a sparse two-angle fan-beam CT geometry, as displayed in Figure~\ref{fig:intro_example}.
In this context, one cannot expect to obtain good reconstructions (which could then be handled by learned regression models) by classical methods.
However, a full reconstruction is not needed, only scalar values for the minimal and maximal thickness are desired.
The idea we follow in this paper is then to circumvent any classical reconstruction and construct an appropriate learned regression model that takes the sparse sinogram data as an input.

\begin{figure}[t]
\centering
\includegraphics[width=0.6225\textwidth]{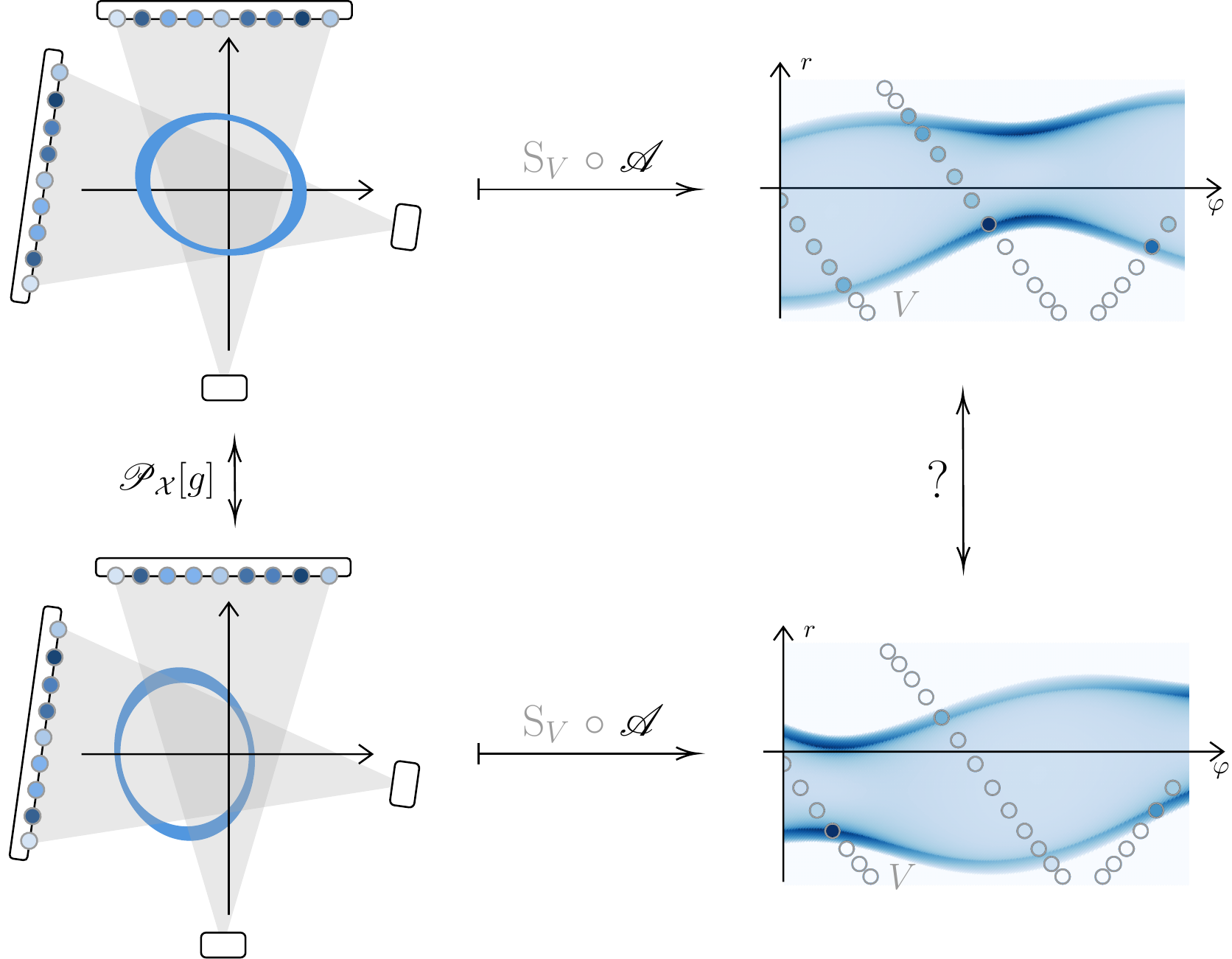}
\caption{The setting of our approach, illustrated for 2-angle fan-beam Radon measurements of tubes: symmetries in $\Y$ (right) are induced by well-known group representations in the physical space $\X$ (left). However, the data is measured only at a finite set of sensor points $V$, depicted by gray circles.}
\label{fig:intro_example}
\end{figure}

On this path, several theoretical and numerical challenges arise.
First of all, to construct equivariant networks, a characterization of sensible symmetry transforms on the measured signals $y$ is required.
Due to the discrete nature of measurements in practice, symmetries often occur only in an approximate way, requiring a modelling scheme that allows to characterize the underlying continuous symmetries.
Finally, one has to construct neural networks that respect these symmetries while operating on the discrete measured data.
In this work, we will tackle all these challenges and demonstrate a successful construction of data-efficient network structures that operate on indirect measurements.

\subsection{Literature Review}
The exploitation of symmetries to improve the generalization performance of deep learning models has been a large field of research during the last years.
Approaches include data augmentation~\cite{LeCun1995}, specialized training schemes or loss functions~\cite{Chen2021, Dittmer2022} and specially tailored equivariant network structures, on which we focus in the following.

Motivated by the success of convolutional neural networks (CNNs), the research on generalized group-equivariant neural network architectures has led to network structures that show promising results on various types of data~\cite{Finzi2020, Kondor2018}.
Group-equivariance is commonly perceived as a strong inductive bias for the construction of data-efficient network architectures.
Its success is believed to rely on two basic properties~\cite{Kondor2018}: On the one hand, equivariance incorporates prior knowledge on the data to provide guarantees for the extrapolation of results to unseen transformed data points and enables parameter-sharing. On the other hand, the non-trivial kernel size of convolutions together with downsampling layers allow to learn local features and combine those features on coarser scales.

\smallskip

In 2016, Cohen and Welling~\cite{Cohen2016} constructed generalized CNNs for image processing that are equivariant with respect to discrete groups.
Since then, several extensions and variants have been proposed.
For example, in~\cite{Cohen2017} the authors develop memory-efficient steerable convolution kernels based on group representation theory.
Lately, Finzi et al.~\cite{Finzi2020} constructed neural networks that are equivariant with respect to arbitrary Lie groups.
To handle the continuous groups on point clouds, they parametrized convolution kernels by small fully-connected networks.
Since we aim to treat arbitrary discretizations, we also follow this approach for the construction of equivariant networks.

In inverse problems, a vast range of classical theory and numerical methods concerned with the reconstruction of $x$ from possibly noisy data $y^\delta = y + \eta^\delta$ exists~\cite{Engl2000, Kaipio2005}.
Symmetries have been a central tool in the development of methods, ranging from steerable wavelets~\cite{Chenouard2012} to regularization by equivariance~\cite{Tang2022}.

During the last years, diverse deep learning methods have shown to improve the quality of reconstructions~\cite{Arridge2019, Baguer2020}.
The incorporation of symmetries in such reconstructions turned out to be a valuable prior and has been studied in various approaches~\cite{Celledoni2021, Chen2021, Dax2022, Tang2022, Chen2023}.

In~\cite{Celledoni2021}, the authors use an equivariant neural network architecture as proximal operator to improve the reconstruction quality obtained by learned iterative schemes.
In this context, possible issues caused by non-equivariant measurement operators that partially regard information are shortly discussed.
On the contrary, the reconstruction approach in~\cite{Chen2021} makes explicit use of this property to design specialized loss functions for a self-supervised learning strategy for end-to-end reconstructions.
In our approach, this effect is handled by the separation of continuous forward measurement operators and their partial evaluation, e.g., due to discretization. 

Other deep learning based reconstruction approaches include symmetry knowledge by the direct use of equivariant networks for post-processing of classical reconstructions~\cite{Winkels2019} or in the construction of learned regularizations in variational methods~\cite{Tang2022}.
To the best of our knowledge, none of the existing works has tackled the problem of directly handling induced symmetries in indirect measurements.

\subsection{Overview of the Paper}

The outline of this paper is as follows.
In Section~\ref{sec:preliminaries} we explain the general setting and introduce the necessary mathematical preliminaries for the analysis of group symmetries in indirect measurements in a compact way.
Section~\ref{sec:math} is the theoretical main part of this paper and contains two mathematical results that investigate the connection between measurement operators and the symmetries that may appear in the measured data.
In particular, Theorem~\ref{thm:equivariance} sharply characterizes all linear operators that translate \mbox{between} given symmetry transforms.
For illustration, we apply this result to the classical Radon operator that models the measurement process in computerized tomography.
Extending on this knowledge, in Section~\ref{sec:layers} we build equivariant neural networks with respect to new types of transforms that have not been tackled before and treat the challenges that appear in the context of applied inverse problems, where only partial measurements are available.
Finally, in Section~\ref{sec:numerics} we apply our constructed networks that can operate directly on the measured data and perform numerical experiments on classification and regression tasks, where we are able to show improved generalization properties when compared to existing methods that are based on learned post-processing of classical reconstructions.

\section{Mathematical Setting and Preliminaries}\label{sec:preliminaries}

\subsection{General Setting}

In this paper, we deal with linear inverse problems $y = \A x$ that can be modelled by a continuous linear forward operator
\begin{equation*}
\A: \X \to \Y
\end{equation*}
mapping between spaces of test functions $\X = \D(\Omega_\X, \K^{n_\X})$ on some domain $\Omega_\X \subseteq \R^{d_\X}$ and $\Y = \D(\Omega_\Y, \K^{n_\Y})$ on some domain $\Omega_\Y \subseteq \R^{d_\Y}$, where $\K \in \{\R, \C\}$.
Since we focus on scalar data in the application, we will abbreviate notation and mostly omit the signal range dimension, i.e. we write $\X = \D(\Omega_\X)$, $\Y = \D(\Omega_\Y)$.
Keep in mind that we implicitly also treat vector field transforms and, hence, kernels are matrix-valued.
In Section~\ref{sec:layers} we will add a discretization scheme to this continuous setting to deal with discrete and potentially noisy measurements.

The ultimate aim of our work is to construct artificial neural networks
\begin{equation*}
\NN_\theta: \Y \to \Z
\end{equation*}
that can be efficiently applied to the measured data and optimized to solve several post-processing tasks while avoiding explicitly solving the inverse problem in a first step.
To this end, we will make use of symmetries in the measurement space $\Y$ that are known in the underlying source space $\X$ and transformed by the forward operator $\A$.

\subsection{Symmetries and Group Representations}

In many cases, continuous local or global symmetries of a source signal $x \in \X$ can be described by means of a known representation~$\P$, i.e., a homomorphism $\P: \G \to \mathrm{GL}(\X)$, of a Lie group $\G$ on $\X$.
To emphasize the underlying space $\X$ we also write $\P \equiv \Px$.
For vector spaces $\V, \W$, we call an operator $\B: \V \to \W$ {\em equivariant} with respect to a representation $\P_\V$ on $\V$ if a representation $\P_\W$ on $\W$ exists such that for all $x \in \V$ and $\g \in \G$ we have
\begin{equation*}
\P_\W[\g]\B(x) = \B(\P_\V[\g]x).
\end{equation*}

A common type of representations $\P$ on function spaces $\V \subseteq \{\R^d \supseteq  \Omega \to \R^n\}$ are so-called {\em domain transforms}, which are induced by a proper Lie group action $\pi$, i.e., a group homomorphism $\p: \G \to \{\Omega \to \Omega\}$ with $\pi[e]= I_d$, acting on the functions $x \in \V$ by (possibly non-linearly) transforming their domain~$\Omega$ so that
\begin{equation*}
(\P[\g]x)(u) = x(\p[\g]^{-1}(u)).
\end{equation*}
Here, a group action is called proper if the mapping $\G \times \Omega \to \Omega \times \Omega$, $(\g,u) \mapsto (\p[\g](u),u)$ is proper, i.e., pre-images of compact sets are compact.

To deal with more general transforms that can appear in measurements, we introduce the notion of {\em generalized domain transforms}, which are representations that can additionally act locally and linearly on the range of the functions via
\begin{equation*}
(\P[\g] x)(u) = p[\g](u) \, x(\p[\g]\inv(u)),
\end{equation*}
where $p: \G \to \mathcal{C}^\infty(\Omega, \mathrm{GL}(\K^n))$ satisfies $p[\g\h](u) = p[\g](u) \, p[\h](\p[\g]\inv(u))$ and $p[e] = I_n$, which is also the sufficient and necessary condition for the induced mapping $\P$ to define a group representation on $\V$.
In this case we also write $\P \equiv \P(\pi,p)$ to emphasize the mappings $\pi: \G \to \{\Omega \to \Omega\}$ and $p: \G \to \mathcal{C}^\infty(\Omega, \mathrm{GL}(\K^n))$.
While representations in the range of functions are also considered for vector fields e.g. in~\cite{Cohen2019}, the symmetries that appear in indirect measurements are new in the way that they can contain local variations also for scalar signals. This behaviour is properly modelled by the introduced class of representations.

All (generalized) domain transforms are primarily characterized by the action of $\p$ on the domain $\Omega$.
If for all $u, u_0 \in \Omega$ there exists $\g \in \G$ such that $\p[\g](u_0) = u$, the set $\Omega$ is called a {\em homogeneous space} of $\p$.
In this case, the domain $\Omega$ can be parametrized by the group elements once an origin $u_0 \in \Omega$ has been fixed.
To this end, we consider the {\em stabilizer subgroup}
\begin{equation*}
\N_\p(u_0) = \{n \in \G \mid \p[n](u_0) = u_0\}.
\end{equation*}
of all elements that keep the origin fixed.
Then, the set $\G/\N_\p(u_0) = \{g\N_\p(u_0) \mid g \in \G\}$ of {\em left cosets} can be identified with the domain, i.e., we have $\G/\N_\p(u_0) \cong \Omega$.
This observation will be an important tool to establish a relation between equivariant operators and convolutions.
For the sake of brevity, in the following we write $\g_u$ for a representative of the group elements that move the origin $u_0$ to $u$, i.e., $\p[\g_u](u_0) = u$, and specify its choice once necessary.

\section{Mathematical Results}\label{sec:math}

As the overall goal of this work is to define equivariant neural networks on the measurements $\A x$, $x \in \X$, it is desirable to describe the symmetries in the measurement space $\Y$.
Given a symmetry transform $\Px$ on the source space $\X$, we consider the transforms that are induced by the measurement, i.e., for all $\g \in \G$, we search for a mapping $\Py[\g]: \range(\A) \to \range(\A)$ that satisfies
\begin{equation}\label{eq:equivariance}
\Py[\g] \A x = \A \Px[\g] x
\end{equation}
for all $x \in \X$.
The following {\em visibility condition} gives a sharp criterion for the existence and properties of such a mapping.

\begin{theorem}[visibility condition]\label{thm:visibility}
For a group representation $\Px$ on $\X$ and a measurement operator $\A: \X \to \Y$ the following statements are equivalent.
\begin{enumerate}[label=(\roman*)]
\item The mapping $\Py[\g]: \range(\A) \to \range(\A)$ given by~\eqref{eq:equivariance} is well-defined for all $\g\in\G$. \label{item:vis_i}
\item Equation~\eqref{eq:equivariance} defines a group representation $\Py: \G \to (\range(\A) \to \range(\A))$ on $\Y$. \label{item:vis_ii}
\item We have $\kernel(\A\Px[\g]) = \kernel(\A)$ for all $\g \in \G$. \label{item:vis_iii}
\end{enumerate}
\end{theorem}

\begin{proof}
\ref{item:vis_i}$\implies$\ref{item:vis_ii}:
Let $\Py: \G \to (\range(\A) \to \range(\A))$ be defined via~\eqref{eq:equivariance}.
Then, $\Py$ is a group homomorphism as, for $y = \A x \in \range(\A)$, we have $\Py[\e] y = \Py[\e] \A x = \A x = y$ and, for arbitrary $\g, \h \in \G$, direct computations show that
\begin{align*}
\Py[\g\h]y &= \Py[\g\h]\A x = \A\Px[\g\h]x = \A\Px[\g]\Px[\h]x\\
&= \Py[\g]\Py[\h]\A x = \Py[\g]\Py[\h]y.
\end{align*}
Moreover, for $\g \in \G$ the mapping  $\Py[\g]: \range(\A) \to \range(\A)$ is linear as for $\lambda_1, \lambda_2 \in \R$ and $y_1 = \A x_1, y_2 = \A x_2 \in \range(\A)$ holds that
\begin{align*}
\Py[\g](\lambda_1 y_1+\lambda_2 y_2) &= \Py[\g]\A(\lambda_1 x_1+\lambda_2 x_2) = \A\Px[\g](\lambda_1 x_1+\lambda_2 x_2)\\
&= \lambda_1\A\Px[\g]x_1+ \lambda_2\A\Px[\g] x_2 = \lambda_1\Py[\g]y_1 + \lambda_2\Py[\g]y_2.
\end{align*}

\ref{item:vis_ii}$\implies$\ref{item:vis_iii}:
Let $\Py: \G \to (\range(\A) \to \range(\A))$ be the above group representation and let $\g \in \G$ such that
\begin{equation*}
\kernel(\A\Px[\g]) \neq \kernel(\A).
\end{equation*}
Then, there is $x \in \kernel(\A)$ with $x \notin \kernel(\A\Px[\g])$ or $\hat{x} \in \kernel(\A\Px[\g])$ with $\hat{x} \notin \kernel(\A)$.
In the first case, however, we have
\begin{equation*}
\A\Px[\g] x = \Py[\g]\A x = 0
\end{equation*}
in contradiction to $x \notin \kernel(\A\Px[\g])$.
In the second case, we similarly get the contradiction
\begin{equation*}
\A\hat{x} = \A\Px[\g\inv]\Px[\g]\hat{x} = \Py[\g\inv]\A\Px[\g]\hat{x}.
\end{equation*}

\ref{item:vis_iii}$\implies$\ref{item:vis_i}:
Let $\kernel(\A\Px[\g]) = \kernel(\A)$ for all $\g \in \G$ and let $y = \A x_1 = \A x_2$ for $x_1, x_2 \in \X$.
Then, $x_1-x_2 \in \kernel(\A) = \kernel(\A\Px[\g])$ and, consequently, for arbitrary $\g \in \G$ follows that
\begin{equation*}
\A\Px[\g]x_1 = \A\Px[\g]x_2
\quad \implies \quad
\Py[\g]\A x_1 = \Py[\g]\A x_2,
\end{equation*}
which shows that $\Py[\g]: \range(\A) \to \range(\A)$ is well-defined.
\end{proof}

The sharp condition $\kernel(\A\Px[\g]) = \kernel(\A)$ for all $\g \in \G$ can be interpreted in terms of the visibility of information in the measurements:
If a feature in the source data can be made visible or invisible in the measurement by the application of a group transform, it is impossible to characterize $\A \Px[\g] x$ only in terms of $x$ and $\g$.
Otherwise, the output can be determined solely from $x$ and $\g$ and, in addition, behaves well in the sense of a group representation structure.
An example for the violation of the visibility condition due to partial measurements is depicted in Figure~\ref{fig:radon_digit}.
Due to the previous theorem, classical reconstructions that do not incorporate any additional assumptions on the data cannot be equivariant with respect to group representations in the source space.
\begin{figure}
\centering
\includegraphics{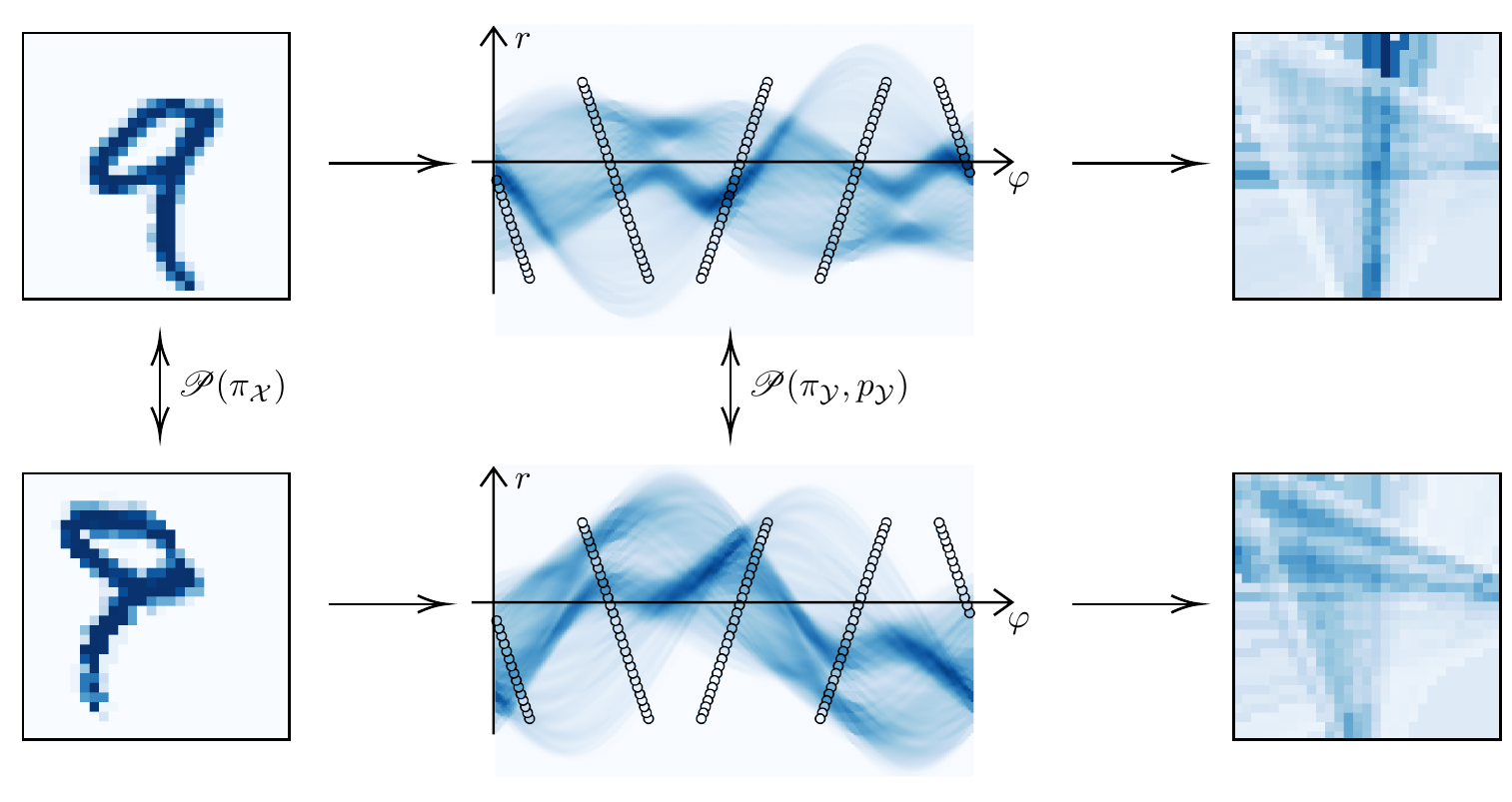}
\caption{Illustration how classical reconstructions (right) of an MNIST digit based on sparse measurements (center, visualized by circles) behave under transformations from $\Aff=\T\rtimes\GL$ (left). In this case, the resulting representation in the full sinograms (center) is a generalized domain transform.}
\label{fig:radon_digit}
\end{figure}
As a consequence, we will focus on applications in which the continuous full measurement operator can be modelled such that it satisfies the visibility condition.
Violations of the visibility condition that are caused by the discretization of the operator are discussed in more detail in Section~\ref{sec:layers}.

In a finite-dimensional setting, the authors of~\cite{Chen2021} state a related condition, where the non-fulfilment of the visibility condition is an explicit prerequisite for their approach to be of advantage for reconstruction.
This is in contrast to our approach, where equivariance is built explicitly into the network structure and requires the fulfilment of the visibility condition.

Our next goal is to exactly characterize the class of measurement operators that transform between certain group representations.
For this purpose, we restrict ourselves to the case of generalized domain transforms as group representation in the source space.
In practice, the representations in the source space are well-known and the notion of generalized domain transforms covers all symmetries that are commonly considered in the known physical domain, cf.~\cite{Cohen2016, Finzi2020, Cohen2019}.
In addition, we assume that the domain $\Omega_\X$ of the input signal is a homogeneous space of the corresponding group action on $\X$.
This assumption, however, can be weakened to obtain a more general but weaker result by treating the domain as a union of orbits, which are homogeneous spaces on their own.
To treat general linear measurement operators with a priori unknown induced symmetries, we do not restrict the representations $\Py$ in any way.

In the following, for a distribution $a \in \D'(\Omega_\X)$ and a function $x \in \D(\Omega_\X)$, we use the notation
\begin{equation*}
a(x) = \int_{\Omega_\X} a(u) \, x(u) \: \diff u.
\end{equation*}
For a group representation $\Px: \G \to (\D(\Omega_\X) \to \D(\Omega_\X))$ we define its {\em adjoint representation} $\Px^\ast: \G \to (\D(\Omega_\X) \to \D(\Omega_\X))$ via
\begin{equation*}
\dint (\Px^\ast[\g]x)(u)\ \varphi(u) \: \diff u = \dint x(u)\ (\Px[\g]\varphi)(u) \: \diff u
\quad \forall \ \varphi \in \D(\Omega_\X)
\end{equation*}
and, with this, for $a \in \D'(\Omega_\X)$ and $g \in \G$ we define $\Px[g]a \in \D'(\Omega_\X)$ via
\begin{equation*}
\Px[g]a(x) = a(\Px^\ast[g]x)
\quad \forall \ x \in \D(\Omega_\X).
\end{equation*}

\begin{theorem}[Equivariant measurement operator is a convolution.]\label{thm:equivariance}
Let $\P_\X \equiv \P_\X(\pi_\X,p_\X)$ be a generalized domain transform such that $\Omega_\X$ is a homogeneous space of $\px$ and let $\A: \D(\Omega_\X) \to \D(\Omega_\Y)$ be a continuous linear measurement operator.
Then, the equivariance condition~\eqref{eq:equivariance} holds if and only if $\A$ is given by a {\em generalized convolution operator}
\begin{equation}\label{eq:convolution}
\begin{split}
(\A x)(v)=&\ (a \conv{\X}{\Y} x)(v)\\
\coloneqq& \dint (\Py[\g_u]\,a)(v)\ p_\X[\g_u\inv](u_0)\,x(u)\ |\det(\Diff\px[\g_u](u_0))|\inv \: \diff u,
\end{split}
\end{equation}
where the distributional kernel $a \in \D'(\Omega_\Y)$ fulfils
\begin{equation}\label{eq:constraint}
(\Py[\g_un]\, a)(v) = (\Py[\g_u]\, a)(v)\ p_\X[n](u_0)\ |\det(\Diff\px[n](u_0))|
\end{equation}
in the sense of distributions for all $n \in \N_\px(u_0)$.
\end{theorem}

Before we proceed to the proof of this theorem, we wish to pass the following remarks.

The condition~\eqref{eq:constraint} guarantees that the representation in equation~\eqref{eq:convolution} is independent of the choice of the representative $\g_u$ of the group elements that move $u_0$ to $u$.

The theorem has several implications for the handling of symmetries in inverse problems.
For instance, in the setting of indirect measurements, it enables us to deduce the induced representations from integral measurement operators and to exactly state the set of operators that translate between these symmetries.

The expression {\em generalized convolution} is motivated by the following remark.

\begin{remark}
If $\Px \equiv \Px(\px)$ and $\Py \equiv \Py(\py)$ are conventional domain transforms, representation~\eqref{eq:convolution} reads
\begin{equation*}
(\A x)(v) = \int_{\Omega_\X} a(\py[\g_u\inv](v))\ x(u)\ |\det(\Diff\px[\g_u](u_0))|\inv \: \diff u.
\end{equation*}
Moreover, condition~\eqref{eq:constraint} reduces to
\begin{equation*}
a(\py[\g_un]\inv v) = a(\py[\g_u]\inv (v)) \,|\det(\Diff\px[n](u_0))|
\quad \forall \ n \in \N_\px(u_0)
\end{equation*}
and, hence, constrains the kernel along the orbits of the stabilizer of $\px$ in $\Y$.
In particular, the special case $\Omega_\X = \R^d = \Omega_\Y$ with origin $u_0 = 0$ and translation group $G = \R^d$ with $\g_u = u$ and group action $\p[g](u) = u+g$ recovers the classical convolution operator
\begin{equation*}
(\A x)(v) = \int_{\R^d} a(v-u)\ x(u) \: \diff u
\end{equation*}
with convolution kernel $a \in \D'(\R^d)$.
\end{remark}

For the construction of equivariant neural networks, Theorem~\ref{thm:equivariance} characterizes the structure of all possible linear layers.
In this context, a special case of this theorem for discrete domain transforms is well-known in the literature and has been proven, e.g., in~\cite{Kondor2018}. 
In~\cite{Cohen2019}, the authors embed a variant of this theorem into a more general and group theoretical framework of Mackey functions for equivariant neural networks, where their notion of transforms on vector fields can be viewed as a realization of our concept of generalized domain transforms.
However, in contrast to this, we provide a rigorous proof in the context of function spaces, where we generalize to distributional kernels, which is more suitable for forward operators in inverse problems.
In addition, compared to existing settings, we do not restrict the group transforms in the output space $\Y$ in any way.
This enables us to treat symmetries of all measurement operators (that satisfy the visibility condition) for known transforms in the source space $\X$.
Example~\ref{ex:radon} below will illustrate the consequences of this theorem for the special case of the classical Radon transform of functions on $\R^2$.

In Section~\ref{sec:layers}, we will construct equivariant linear layers that are suited for the application to indirect measurements and characterized by our above theorem, which is formulated in the continuous setting.

\begin{proof}[Proof of Theorem \ref{thm:equivariance}]
In the following, for the sake of brevity, we set
\begin{equation*}
\omega[\g](u) = |\det(\Diff\px[\g](u))|
\end{equation*}
for $\g \in \G$ and $u \in \Omega_\X$.

For the proof of the 'if'-direction let $a \in \D'(\Omega_\X)$ satisfy~\eqref{eq:constraint}.
Then, we can choose any representative $\g_u \in \g_u\N_{\px}(u_0)$ and define $\A: \D(\Omega_\X) \to \D(\Omega_\Y)$ via~\eqref{eq:convolution}.
To verify~\eqref{eq:equivariance}, we compute the image of a transformed $x \in \D(\Omega_\X)$ under $\A$ by a change of variables in the distribution and use the fact that $\omega[\g\h](u) = \omega[g](\px[h](u))\ \omega[\h](u)$.
Then, direct calculations indeed show that
\begin{align*}
(a \conv{\X}{\Y} \Px[\g]x)(v) &= \int_{\Omega_\X}(\Py[\g_u]\,a)(v)\ (\Px[\g_u\inv\g]x)(u_0)\ \omega[\g_u](u_0)\inv\,\diff u\\
&= \int_{\Omega_\X}(\Py[\g\g_u]\,a)(v)\ (\Px[\g_u\inv]x)(u_0)\ \omega[\g\g_u](u_0)\inv\,\omega[\g](u)\,\diff u\\
&= \int_{\Omega_\X}(\Py[\g]\Py[\g_u]\,a)(v)\ (\Px[\g_u\inv]x)(u_0)\ \omega[\g_u](u_0)\inv\,\diff u \\
&= (\Py[\g](a \conv{\X}{\Y} x))(v).
\end{align*}

For the proof of the 'only if'-direction let $\A: \D(\Omega_\X) \to \D(\Omega_\Y)$ satisfy the equivariance condition~\eqref{eq:equivariance}. 
Due to the continuity of $\A$ and the Schwartz kernel theorem~\cite{Hormander1990}, we have the representation
\begin{equation*}
(\A x)(v) = \dint A(u,v) \, x(u) \: \diff u,
\end{equation*}
where $A(\cdot,v) \in \D'(\Omega_\X)$ is a distribution for every $v \in \Omega_\Y$.
Let $(\delta_k)_{k \in \mathbb{N}}$ be a Dirac sequence of smooth functions $\delta_k \in \D(\Omega_\X)$ and, for $k \in \mathbb{N}$, consider the approximate kernel
\begin{equation*}
A_k(w,v) = \dint A(u,v)\ (\Px[\g_w]\,\delta_k)(u)\ p_\X[\g_w\inv](u_0)\ \omega[\g_w](u_0)\inv \: \diff u,
\end{equation*}
which satisfies $A_k(\cdot, v) \in \D(\Omega_\X)$ for all $v \in \Omega_\Y$ as well as $A_k(u,\cdot) \in \D(\Omega_\Y)$ for all $u \in \Omega_\X$.
The induced operator $\A_k: \D(\Omega_\X) \to \D(\Omega_\Y)$ can be rewritten as
\begin{align*}
(\A_k x)(v) &= \dint A_k(w,v) \, x(w) \: \diff w\\
&= \dint\dint A(u,v)\ (\Px[\g_w]\,\delta_k)(u)\ p_\X[\g_w\inv](u_0)\ \omega[\g_w](u_0)\inv \: \diff u \: x(w) \: \diff w\\
&= \dint A(u,v) \dint (\Px[\g_w]\,\delta_k)(u)\ p_\X[\g_w\inv](u_0)\ x(w)\,\omega[\g_w](u_0)\inv \: \diff w \: \diff u,
\end{align*}
where we use Fubini's theorem for distributions~\cite[Theorem 40.4]{Treves1967}.
By considering
\begin{equation*}
(\delta_k \conv{\X}{\X} x)(u) = \dint (\Px[\g_w]\,\delta_k)(u)\ p_\X[\g_w\inv](u_0)\ x(w)\,\omega[\g_w](u_0)\inv \: \diff w,
\end{equation*}
we obtain the representation
\begin{equation*}
\A_k x = \A(\delta_k \conv{\X}{\X} x).
\end{equation*}
Standard computations show the convergence $\delta_k \conv{\X}{\X} x \to x$ in $\D(\Omega_\X)$ and by continuity of $\A$ follows that
\begin{equation*}
\A_k x \xrightarrow{k \to \infty} \A x.
\end{equation*}
Using the equivariance property~\eqref{eq:equivariance} of $\A$ and the equivariance of $\conv{\X}{\X}$ that we have already shown in the first part of this proof, we also have
\begin{equation*}
\A_k \Px[\g] x = \A\, (\delta_k\conv{\X}{\X}\Px[\g] x) = \A\Px[\g]\,(\delta_k\conv{\X}{\X} x) = \Py[\g]\A\,(\delta_k\conv{\X}{\X} x) = \Py[\g] \A_k x
\end{equation*}
and a change of variables implies that
\begin{equation*}
\dint A_k(\px[\g](u),v)\ p_\X[\g](\px[\g](u))\ x(u)\,\omega[\g](u) \: \diff u = \dint (\Py[\g]A_k)(u,v)\,x(u) \: \diff u.
\end{equation*}
Since this holds true for all $x \in \D(\Omega_\X)$, we obtain for all $\g \in \G$ the condition
\begin{equation*}
A_k(u,v) = (\Py[\g]\inv A_k)(\px[\g](u),v)\ p_\X[\g](\px[\g](u))\ \omega[\g](u).
\end{equation*}
By choosing $\g = \g_u\inv$ and defining $a_k := A_k(u_0, \cdot) \in \D(\Omega_\Y)$, we arrive at
\begin{equation*}
A_k(u,v) = (\Py[\g_u]\, a_k)(v)\ p_\X[\g_u\inv](u_0)\ \omega[\g_u](u_0)\inv,
\end{equation*}
applying the inverse function rule in $\omega$.
Since this holds true for all choices $\g_u \in \g_u\N_{\px}(u_0)$, the $a_k$ have to fulfil
\begin{equation}\label{eq:constraint_approx}
\begin{split}
(\Py[\g_un]\, a_k)(v)\ &= (\Py[\g_u]\, a_k)(v)\ p_\X[\g_u\inv](u_0) \ p_\X[(\g_un)\inv](u_0)\inv\ |\det(\Diff\px[n](u_0))|\\
&= (\Py[\g_u]\, a_k)(v)\ p_\X[n](u_0)\ |\det(\Diff\px[n](u_0))|.
\end{split}
\end{equation}
for all $n \in \N_{\px}(u_0)$.
Finally, taking the limit $k \to \infty$ gives
\begin{align*}
(\A x)(v) &= \lim_{k\to\infty} \dint (\Py[\g_u]\, a_k)(v)\ p_\X[\g_u](u)\inv\ x(u)\ \omega[\g_u](u_0)\inv \: \diff u\\
&= \dint (\Py[\g_u]\, a)(v)\ p_\X[\g_u](u)\inv\ x(u)\ \omega[\g_u](u_0)\inv \: \diff u,
\end{align*}
where $a \in \D'(\Omega_\Y)$ is the limit of $(a_k)_{k \in \mathbb{N}}$ in $\D'(\Omega_\Y)$ and $u \mapsto (\Py[\g_u]\, a)(v)$ has to be understood as a distribution in $\D'(\Omega_\X)$ for each $v \in \Omega_\Y$.
Similarly, we obtain the condition~\eqref{eq:constraint} for $a$ by taking the limit $k \to \infty$ in~\eqref{eq:constraint_approx} in distributional sense.
\end{proof}

To close this section we now discuss the Radon transform as an illustrating example.
\begin{example}\label{ex:radon}
For $\Omega_\X = \R^2$ and $\Omega_\Y = \R \times [0,2\pi)$ the Radon operator ${\Radon: \D(\Omega_\X) \to \D(\Omega_\Y)}$ models the measurement process in computerized tomography, cf.~\cite{Natterer2001}, and is defined as
\begin{equation*}
(\Radon x)(r,\varphi) = \int_{\R^2} \delta(u^T \vec{\varphi} - r) \, x(u) \: \diff u,
\end{equation*}
integrating over the line $\{u \in \R^2 \mid u^T\vec{\varphi} = r\}$ with normal direction $\vec{\varphi} = (\cos(\varphi),\sin(\varphi))^T \in \Sphere^1$.
We consider the roto-translation group $G = \SE = \T \rtimes \SO$, which we parametrize by $(s,\gamma) \in \R^2 \times [0,2\pi)$ with operation
\begin{equation*}
(s_1,{\gamma_1})\cdot(s_2,{\gamma_2})=(s_1+R(\gamma_1)s_2, {\gamma_1+\gamma_2}),
\end{equation*}
where addition of angles is performed modulo $2\pi$ and $R(\gamma) \in \SO$ is the rotation matrix
\begin{equation*}
R(\gamma) = \begin{pmatrix}\cos\gamma & -\sin\gamma\\ \sin\gamma & \cos\gamma\end{pmatrix}.
\end{equation*}
A group action $\px$ on $\Omega_\X = \R^2$ can be defined via
\begin{equation*}
\px[(s,{\gamma})](u)=R(\gamma)u+s
\end{equation*}
so that $\Omega_\X$ is a homogeneous space of $\px$ and fixing the origin $u_0 = 0 \in \Omega_\X$ we obtain
\begin{equation*}
\N_\px(u_0) = \{(0,{\gamma}) \mid \gamma \in [0,2\pi)\}.
\end{equation*}
Straightforward computations show that
\begin{equation*}
(\Radon(x\circ\px[(s,\gamma)^{-1}]))(r,\varphi) = (\Radon x)(r - s^T \vec{\varphi},\varphi-\gamma)
\end{equation*}
and, hence, we define the group action $\py$ on $\Omega_\Y = \R \times [0,2\pi)$ via
\begin{equation*}
\py[(s,\gamma)](r,\varphi) = (r+s^T(\overrightarrow{\varphi+\gamma}), \varphi+\gamma).
\end{equation*}

\bigbreak

With this, $\Omega_\Y$ is a homogeneous space of $\py$.
For $u \in \Omega_\X$ we choose $\g_u = (u,0) \in G$ and obtain $\py[\g_u\inv](r,\varphi) = (r-u^T\vec{\varphi},\varphi)$ as well as $\det(\Diff\px[\g_u](u_0)) = 1$.
As a result, Theorem~\ref{thm:equivariance} implies that any bounded linear operator $\A: \D(\Omega_\X) \to \D(\Omega_\Y)$ that transforms translations and rotations exactly like the Radon operator $\Radon$ can be written as
\begin{equation*}
(\A x)(r,\varphi) = \int_{\R^2} a(r-u^T\vec{\varphi},\varphi) \, x(u) \: \diff u
\end{equation*}
with $a \in \mathcal{D}'(\Omega_\Y)$.
Furthermore, using $(\g_u n)\inv = (-R(-\gamma)u,-\gamma)$ for $n = (0,\gamma) \in \N_\px(u_0)$, the condition on $a$ reads
\begin{equation*}
a(r-(R(-\gamma)u)^T(\overrightarrow{\varphi-\gamma}),\varphi-\gamma) = a(r-u^T\vec{\varphi},\varphi)
\iff
a(r-u^T\vec{\varphi},\varphi-\gamma) = a(r-u^T\vec{\varphi},\varphi)
\end{equation*}
for arbitrary $\gamma \in [0,2\pi)$ so that $a \equiv a(r)$ is independent of the angle $\varphi \in [0,2\pi)$.
Therefore, all possible operators are given by
\begin{equation*}
(\A x)(r,\varphi) = \int_{\R^2} a(r-u^T\vec{\varphi}) \, x(u) \: \diff u
\end{equation*}
for some distribution $a \in \D'(\R)$ and we obtain the Radon operator itself for $a = \delta$.

To also include anisotropic scaling, we consider the group $G = \Aff = \T \rtimes \GL$ of all invertible affine transforms with positive determinant, for which we define the group operation via
\begin{equation*}
(s_1, A_1) \cdot (s_2,A_2) = (s_1 + A_1 s_2, A_1 A_2)
\end{equation*}
and the group representation $\Px$ on $\X$ via
\begin{equation*}
(\Px[(s,A)] x)(u) = x(\px[(s,A)]^{-1}(u)).
\end{equation*}
with
\begin{equation*}
\px[(s,A)](u) = Au + s.
\end{equation*}
Then, $\Omega_\X$ is a homogeneous space of $\px$ and fixing the origin $u_0 = 0 \in \Omega_\X$ we obtain
\begin{equation*}
\N_\px(u_0) = \{(0,A) \mid A \in \GL\}.
\end{equation*}
Standard computations show that
\begin{equation*}
(\Radon\Px[(s,A)]x)(r,\varphi) = \int_{\R^2}\delta(u^T\Vec{\varphi}-r)\ x(A\inv(u-s)) \: \diff u = \frac{\det(A)}{\alpha(A,\varphi)} \, (\Radon x)\Bigl(\frac{r-s^T\Vec{\varphi}}{\alpha(A,\varphi)}, \theta(A,\varphi)\Bigr)
\end{equation*}
with
\begin{equation*}
\alpha(A,\varphi) = \|A^T\Vec{\varphi}\|_2
\quad \mbox{ and } \quad
\theta(A,\varphi) = \arccos\bigg(\frac{(A^T\Vec{\varphi})_1}{\alpha(A,\varphi)}\bigg)
\end{equation*}
so that
\begin{equation*}
\Vec{\theta}(A,\varphi) = \frac{A^T\Vec{\varphi}}{\|A^T\Vec{\varphi}\|_2}.
\end{equation*}
Hence, we define the group representation $\Py$ on $\Y$ via
\begin{equation*}
(\Py[(s,A)]y)(r,\varphi) = p_\Y[(s,A)](r,\varphi)\ y(\py[(s,A)]\inv(r,\varphi))
\end{equation*}
where
\begin{equation*}
p_\Y[(s,A)](r,\varphi) = \frac{\det(A)}{\alpha(A,\varphi)}
\quad \mbox{ and } \quad
\py[(s,A)](r,\varphi) = \Bigl(\frac{r+s^T A\invT\Vec{\varphi}}{\alpha(A\inv,\varphi)}, \theta(A\inv,\varphi)\Bigr)
\end{equation*}
so that $\Py$ is a generalized domain transform and
\begin{equation*}
\py[(s,A)]\inv(r,\varphi) = \Bigl(\frac{r-s^T\Vec{\varphi}}{\alpha(A,\varphi)}, \theta(A,\varphi)\Bigr).
\end{equation*}
For illustration, Figure~\ref{fig:radon_digit} shows the action of $\Py$ on sinograms of handwritten digits.
For $u \in \Omega_\X$ we choose $\g_u = (u,I_2) \in G$ and obtain $p_\Y[\g_u](r,\varphi) = 1$, $\py[\g_u]\inv(r,\varphi) = (r-u^T\vec{\varphi},\varphi)$ as well as $\det(\Diff\px[\g_u](u_0)) = 1$.
Consequently, Theorem~\ref{thm:equivariance} implies again that any bounded linear operator $\A: \D(\Omega_\X) \to \D(\Omega_\Y)$ that transforms invertible affine transforms with positive determinant exactly like the Radon operator $\Radon$ can be written as
\begin{equation*}
(\A x)(r,\varphi) = \int_{\R^2} a(r-u^T\vec{\varphi},\varphi) \, x(u) \: \diff u
\end{equation*}
with $a \in \mathcal{D}'(\Omega_\Y)$.
Furthermore, using $(\g_u n)\inv = (-A\inv u,A\inv)$ for $n = (0,A) \in \N_\px(u_0)$, the condition on $a$ reads
\begin{align*}
\frac{\det(A)}{\|A^T\Vec{\varphi}\|_2} \, a\Bigl(\frac{r-u^T\Vec{\varphi}}{\|A^T\Vec{\varphi}\|_2}, \theta(A,\varphi)\Bigr) & = a(r-u^T\vec{\varphi},\varphi) \det(A) \\
& \iff a\Bigl(\frac{r-u^T\Vec{\varphi}}{\|A^T\Vec{\varphi}\|_2}, \theta(A,\varphi)\Bigr) = \|A^T\Vec{\varphi}\|_2 \, a(r-u^T\vec{\varphi},\varphi)
\end{align*}
for arbitrary $A \in \GL$ so that $a \equiv a(r)$ is independent of the angle $\varphi \in [0,2\pi)$.
Moreover, $a \in \D'(\R)$ is positively homogeneous of degree $-1$ and, according to~\cite[Chap.~1, §~3.11]{Gelfand1964}, there are constants $c_1,c_2 \in \R$ such that
\begin{equation*}
a = c_1 \, \chi^{-1} + c_2 \, \delta,
\end{equation*}
where
\begin{equation*}
\chi^{-1}(\psi) = \int_\R \frac{\psi(r) - \psi(0)}{r} \: \diff r
\quad \mbox{ for } \psi \in \D(\R).
\end{equation*}

In summary, the symmetries of the Radon operator characterize most of its behaviour and leave only very few degrees of freedom for operators that share the same symmetry properties.
This might become useful, e.g., in the construction of end-to-end reconstruction schemes, which could benefit from an encoded reverse conversion of symmetries.
\end{example}

\section{Equivariant Layers}\label{sec:layers}

Now that we have characterized the group representation $\Py$ on the measurements in $\Y$, we aim for the construction of equivariant neural networks $\NN_\theta: \Y \to \Z$ such that
\begin{equation*}
\Pz[\g]\,\NN_\theta(y) = \NN_\theta(\Py[\g]\, y)
\quad \forall \, y \in \Y, ~ \g \in \G.
\end{equation*}
Here, the space $\Z$ and representation $\Pz$ have to be chosen depending on the application. In reconstruction or semantic segmentation tasks, the natural choice is $\Z = \X$ and $\Pz = \Px$. Learned preprocessing could be obtained by $\Z = \Y$ and $\Pz = \Py$. Invariant classification tasks require $\Pz[\g] = \mathrm{Id}$ independently of $\g \in \G$.
Based on the findings in \cite{Cohen2016}, it is however useful to construct invariance only in the last layer of the network and use more complex group representations in the inner hidden part.

In the continuous setting studied previously, the construction of equivariant linear neural network layers for generalized domain transforms is completely characterized by Theorem~\ref{thm:equivariance}.
In practical applications, however, we do not only need a discretization for numerical purposes but also have to deal with a finite and possibly very limited number of measurement points.
We therefore model the whole indirect measurement process as $y_V^\delta = \S \A x + \eta^\delta$, where $\eta^\delta$ models noise in the measurements of noise level $\delta$ and $\S$ defines a sampling scheme
\begin{equation*}
\S:\Y \to \Yv
\end{equation*}
evaluating a continuous signal $y \in \Y$ at discrete measurement points $V = \{v_i\}_{i=1,...,n} \subset \Omega_\Y$.

In this section, we will construct neural networks that are (approximately) equivariant to generalized domain transforms as representations in the input and output of all layers and that can be applied directly to arbitrary data generated by indirect measurements. Figure \ref{fig:space_scheme} gives an overview of the setting and the involved spaces.

\begin{figure}[t]
\centering
\includegraphics[width=\textwidth]{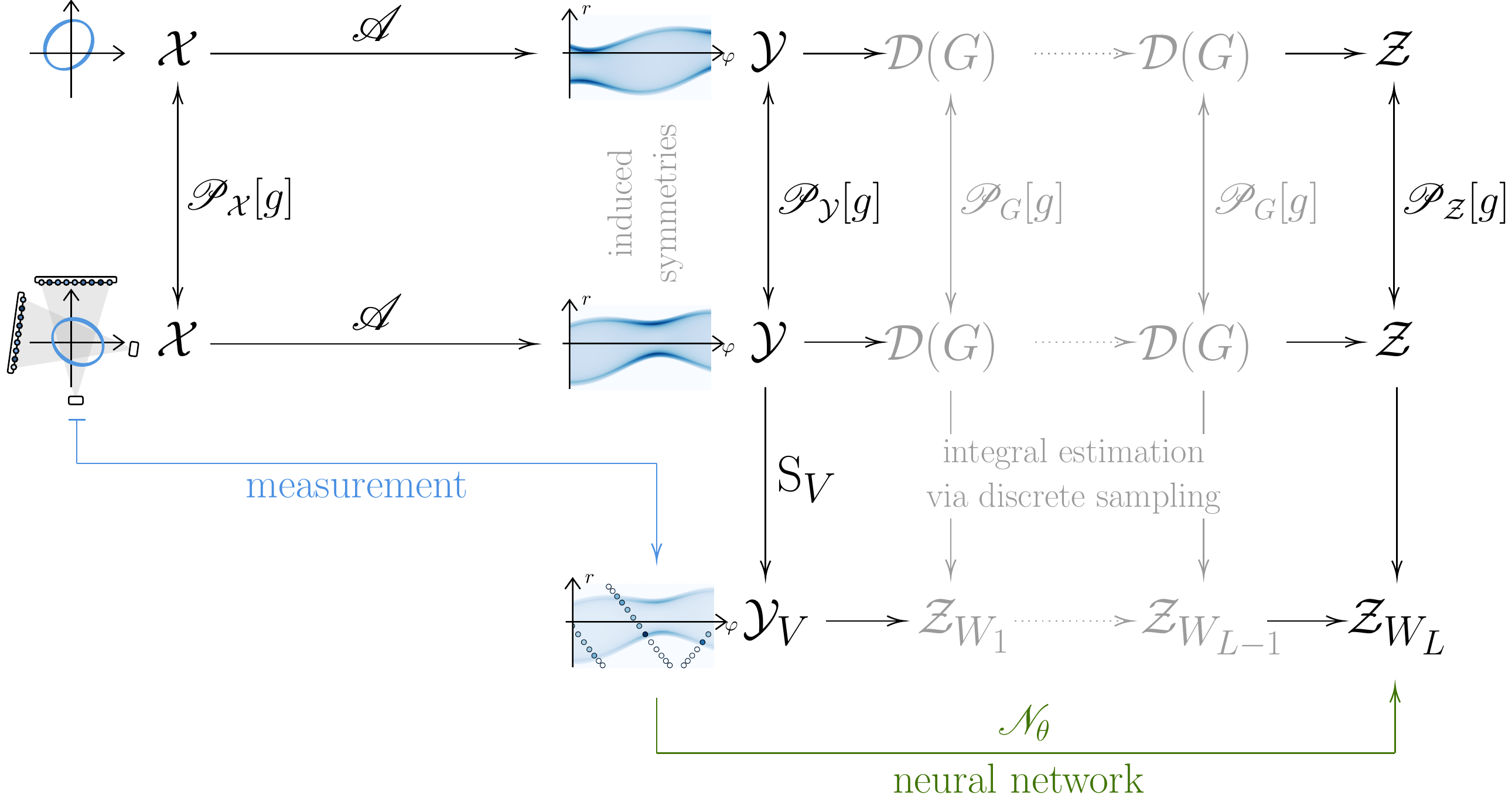}
\caption{Overview of our approach, including the measurement scheme as well as the learned processing by an equivariant neural network. The plots illustrate the modelling and symmetries of a two-angle fan-beam Radon operator on a rotated and translated elliptical ring.}
\label{fig:space_scheme}
\end{figure}

\subsection{The Challenge of Discrete Equivariance}

When trying to treat equivariance in discrete measurements, the first problem that occurs is the definition of appropriate group representations on the discrete data.
Since the data is only partially captured, the discretized operators violate the visibility condition in Theorem~\ref{thm:visibility}.
As a result, it is not clear how general group transforms in the input can be handled in the measurements. In addition, also classical reconstruction methods applied to sparse measurements cannot be exactly group equivariant, which limits the performance of all methods that are based on reconstruction schemes.

A possible way out that can also be implemented numerically is to be satisfied with equivariance only with respect to a discrete subgroup $\G_V \subset \G$, which would result in an approach similar to~\cite{Kondor2018} and yields a direct generalization of classical CNNs.
However, this requires the measurements to lie on a grid that fits the structure of the subgroup.
More precisely, to handle the measurement points $V$ it is necessary that we have the set equality
\begin{equation*}
\py[\G_V](v_0) = V
\end{equation*}
for a suitably chosen $v_0 \in \Omega_\Y$.
In general, such a subgroup $\G_V$ does not exist and, hence, this idea can only be applied for very specific sampling schemes $\S$.
A numerically more complex but also more general solution is to construct networks that are {\em approximately} equivariant to the representation $\Py$ on the whole Lie group $\G$ and can operate on arbitrary point clouds, extending the work in~\cite{Finzi2020}.
Following this approach, only the continuous measurement operator $\A$ has to be equivariant, i.e., fulfil the visibility condition in Theorem~\ref{thm:visibility}.
For many applications (e.g. sparse CT or image inpainting), this can be achieved by a proper modelling of $\A$ and $\S$ such that the definition of $\A$ reveals the underlying symmetries that would appear in the measurement data analytically if one was able to measure the complete signal.
However, discrete measurements also induce limitations which cannot be overcome by any advanced approach (without additional data knowledge): the equivariance of the resulting methods in the limit and its improvement for an increasing number of measurement points is constrained by the uniformity of the sensor locations.
Nevertheless, we expect the resulting networks to provide an efficient architecture also in case of non-uniform sparse measurements.

\subsection{Construction of Equivariant Networks}

Neural networks are typically built based on linear layers and the localized application of non-linear activation functions.
To construct equivariant networks, our approach is to concatenate equivariant layers $\NN^\ell: \Z_{\ell-1} \to \Z_\ell$ and set
\begin{equation*}
\NN_\theta = \NN^{N_L} \circ \dots \circ \NN^1
\end{equation*}
where $\Z_0 = \Y$ and $\Z_L = \Z$.
As a consequence of Theorem \ref{thm:equivariance}, for given input and output generalized domain transforms, the structure of an equivariant linear layer is known to be convolutional and the layer can be fully parametrized by the (distributional) kernel $a \equiv a_\theta$.

For the implementation of such linear layers we have to approximate the integral in~\eqref{eq:convolution}, where we solely consider generalized domain transforms for the input and output of the layers.
Such transforms, however, cover the representations in $\Y$ or $\Z$ that may appear in our desired indirect measurement applications.
As usual for convolutional networks, we use localized convolution kernels with compact support around the origin.
We approximate the integral for linear layer $\NN^{\ell+1}$ by
\begin{align*}\label{eq:conv_layer}
(\NN^{\ell+1} z)(w'_j) &=  \int_{\Omega_{\Z_\ell}} (\P_{\Z_{\ell+1}}[\g_{w}]\,a_\theta)(w'_j)\ p_{\Z_{\ell}}[\g_{w}\inv](w_0)\,z(w)\  |\det(\Diff\p_{\Z_{\ell}}[\g_{w}](w_0))|\inv \: \diff w\\
&\approx \sum_{i = 1}^{M_{\ell}} c_i \, (\P_{\Z_{\ell+1}}[\g_{w_i}]\,a_\theta)(w_j')\ p_{\Z_{\ell}}[\g_{w_i}\inv](w_0)\,z({w_i})\ |\det(\Diff\p_{\Z_{\ell}}[\g_{w_i}](w_0))|\inv,
\end{align*}
for the parameterized continuous kernel $a_\theta$ and collocation points $W_\ell = \{w_i\}_{i=1,...,M_\ell} \subset \Omega_{\Z_\ell}$, where $z \in \Z_\ell$, $w'_j \in \Omega_{\Z_{\ell+1}}$ and $C_\ell = \{c_i\}_{i=1,...,M_\ell} \subset \K$ denote quadrature weights.

In the inner layers of the network, the $w_i$ are sampled uniformly, such that we perform Monte-Carlo summation with constant $c_i=|\Omega_{\Z_\ell}|/M_\ell$.
In the first layer, the input is known only at the measurement locations and, therefore, we have to choose $W_0 = V$ and perform numerical integration with fixed quadrature points.
For this purpose, we build on the recent work in \cite{Doherty2022}, where the authors report successful usage of learned quadrature weights, i.e., the $c_i$ become parameters during the training of the network. Even though the approximation of the integrals may induce numerical errors, we expect that the resulting convolutional layer structure forms a good inductive bias in all layers while this scheme enables us to handle arbitrary measurement geometries. The inputs then consist of pairs $(y_i, v_i)_{v_i \in V}$, where $y_i$ corresponds to the measured value at node $v_i$ (e.g., $v_i = (r_i, \varphi_i)$ in the case of a sinogram). 

The constraint~\eqref{eq:constraint} for the kernel along the orbits of the stabilizer can however be quite restrictive and limit the expressive power of the layers.
For the symmetries that appear in indirect measurements, the limitations are especially severe, since these orbits can become unbounded, preventing localized kernels.
To overcome this, one has to properly choose the group representations used throughout the neural network.
In most existing discrete~\cite{Kondor2018} and continuous~\cite{Finzi2020} approaches, this is tackled by choosing $\Z_\ell = \D(\G, \R^{c_\ell})$, $\ell = 1,\dots,L-1$, and the canonical domain transform $\P_\G = \P(\pg)$, $\pg: \G \to \G$, $\g \mapsto \g\cdot$, on the group itself as the input and output representation of inner layers.
As a homogeneous space of itself, $\G$ has a trivial stabilizer and the constraint~\eqref{eq:constraint} vanishes.

In the first layer, $\NN^1: \Y \to \D(\G, \R^{c_1})$, Theorem~\ref{thm:equivariance} gives
\begin{equation*}
(\NN^1 y)(\g) = \int_{\Omega_\Y} a_\theta(\g_v\inv\g)\ p_\Y[\g_v\inv](v_0)\,y(v)\ |\det(\Diff\py[\g_v](v_0))|\inv \: \diff v,
\end{equation*}
where the kernel $a_\theta \in \D'(\G)$ has to fulfil
\begin{equation*}
a_\theta(n\inv\g) = a_\theta(\g)\ p_\Y[n](v_0)\ |\det(\Diff\py[n](v_0))|.
\end{equation*}
This can be realized by setting
\begin{equation*}
a_\theta(\g)=\Tilde{a}_\theta(\py[\g\inv](v_0))\ p_\Y[\g](v_0)\inv\ |\det(\Diff\py[\g\inv](v_0))|
\end{equation*}
for a kernel $\Tilde{a}_\theta \in \D'(\Omega_\Y)$ so that
\begin{equation*}
(\NN^1 y)(\g) = \int_{\Omega_\Y} \Tilde{a}_\theta(\py[\g\inv](v))\ p_\Y[(\g\inv)](v_0))\,y(v)\ |\det(\Diff\py[\g](v_0))|\inv \: \diff v.
\end{equation*}

As usual for neural networks, we construct non-linear layers using local non-linearities that act only on the values of activations and are independent of the spatial position.
Since, in most cases, activation functions are applied only in the inner parts of the network (where all representations $\P_{\Z_\ell}$ are given by the canonical domain transform on $\G$), we do not have to consider additional constraints on these functions.

\section{Numerical Experiments}\label{sec:numerics}

\begin{figure}[t]
\centering
\includegraphics[width = 0.8\textwidth]{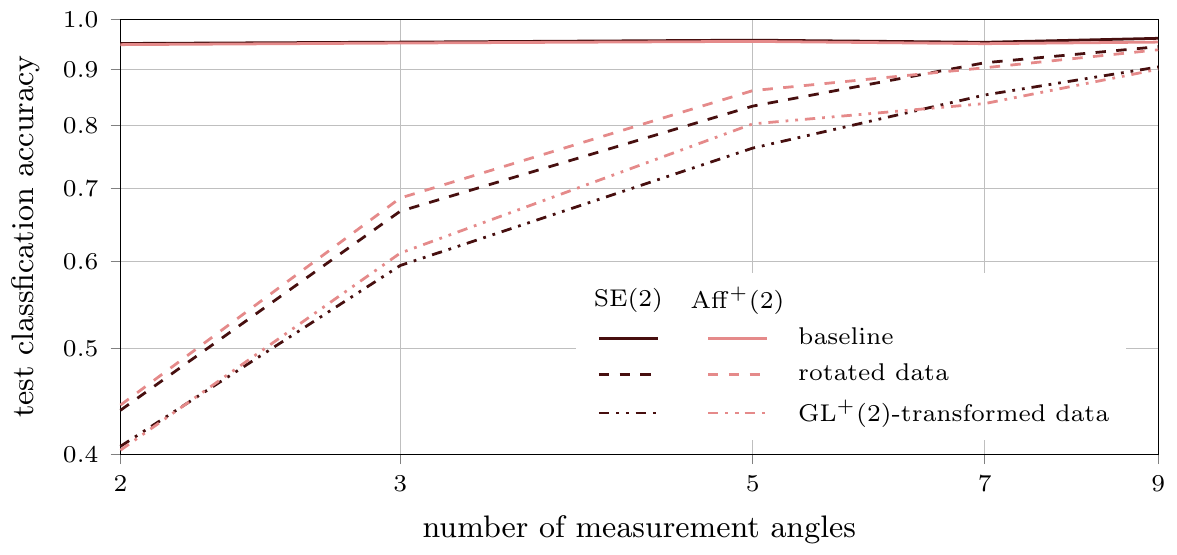}
\caption{Classification accuracy for fan-beam CT measurements of randomly rotated or $\GL$-transformed MNIST test dataset for different neural networks that are trained exclusively on upright digits.}
\label{fig:invariance}
\end{figure}

To obtain first insights on the effectiveness of our idea in practice, we implemented the previously defined layers and performed a series of experiments in the setting of classification and regression tasks based on (sparse) CT measurements,
where we require equivariance/invariance from $\X$ to $\Z$, while measurements are taken in $\Y$.

\subsection{Implementation}

We implemented a flexible PyTorch~\cite{PyTorch} framework that allows us to compute the generalized convolution in~\eqref{eq:conv_layer} for different input and output representations.
Our code is publicly available on GitHub\footnote{\url{https://github.com/nheilenkoetter/equivariant-nn-for-indirect-measurements}}.

In each layer, we use a small fully-connected neural network with two layers, batch normalization~\cite{Ioffe2015} and Swish activations~\cite{Ramachandran2017} to parameterize a basis of the convolution kernels $a_\theta$.
The kernels connecting all input and output channels are then additionally parameterized by individual linear factors.
In this setting, we use the automatic optimization of Einstein summation in PyTorch to perform the efficient PointConv trick~\cite{Finzi2020}.
Where the constraint~\eqref{eq:constraint} allows, we enforce locality of the kernels $a_\theta$ respectively $\Tilde{a}_\theta$ by an additional rapidly decaying factor, i.e.,
$a_\theta(w_j') = \exp(-d(w_j', w_0')^2/r^2)\ \Hat{a}_\theta(w_j')$, where $d: \Omega_{\Z_\ell} \times \Omega_{\Z_\ell} \to \R_{\geq 0}$ is a metric that is invariant with respect to $\pi_{\Z_\ell}$.
In addition, we compute the resulting sum only over a set of $k = 27$ nearest neighbours to minimize the computational effort.

We use a residual network architecture~\cite{He2015} that is built of three residual blocks connected by downsampling layers (increasing the field of view and decreasing the amount of sampled points).
In each block, we use three convolutional layers combined with batch normalization and ReLU activations~\cite{Fukushima1975}.
We choose $22$ channels in the initial layer and double the amount of channels in each downsampling step, while we halve the amount of sampled points.
We obtain invariance by adding a global average pooling in the last layer, leaving us only with the constructed channel dimension.

In every iteration of the optimization, we sample a new set of inner collocation points. However, to reduce the computational effort, we sample only once for each batch and each residual block and use the same points along all layers with identical scale.

In all of our experiments on image data, we simulate fan-beam Radon measurements using the ASTRA~\cite{vanAarle2016, vanAarle2015} toolbox.
For optimization, we use Adam~\cite{Kingma2014} with weight decay and a learning rate scheduling scheme.

\begin{figure}[t]
\centering
\subfigure[Test on RotMNIST]{\includegraphics[width=0.425\textwidth]{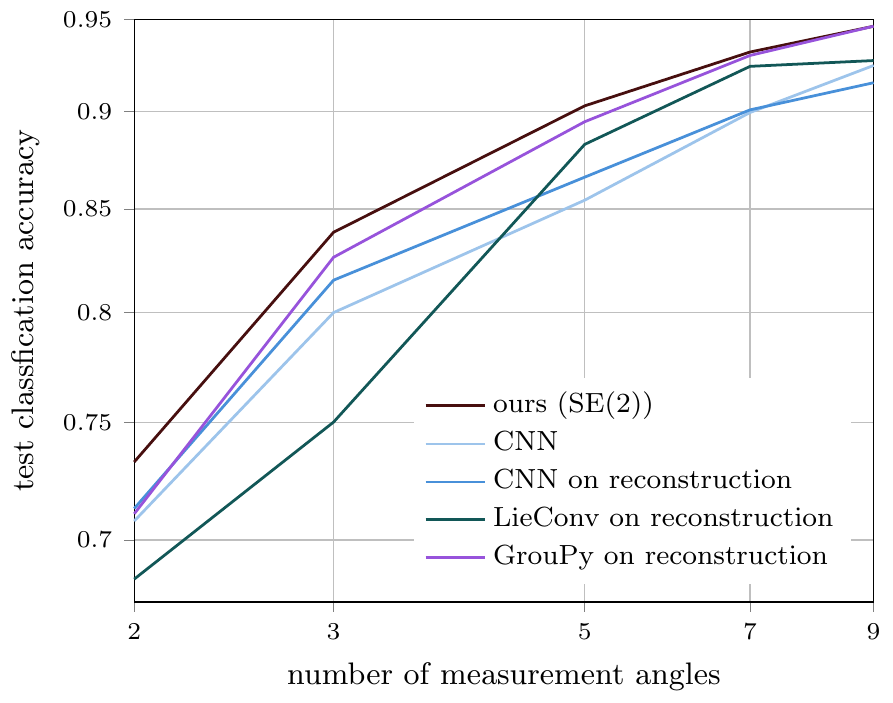}}
\hfil
\subfigure[Test on LinMNIST]{\includegraphics[width=0.425\textwidth]{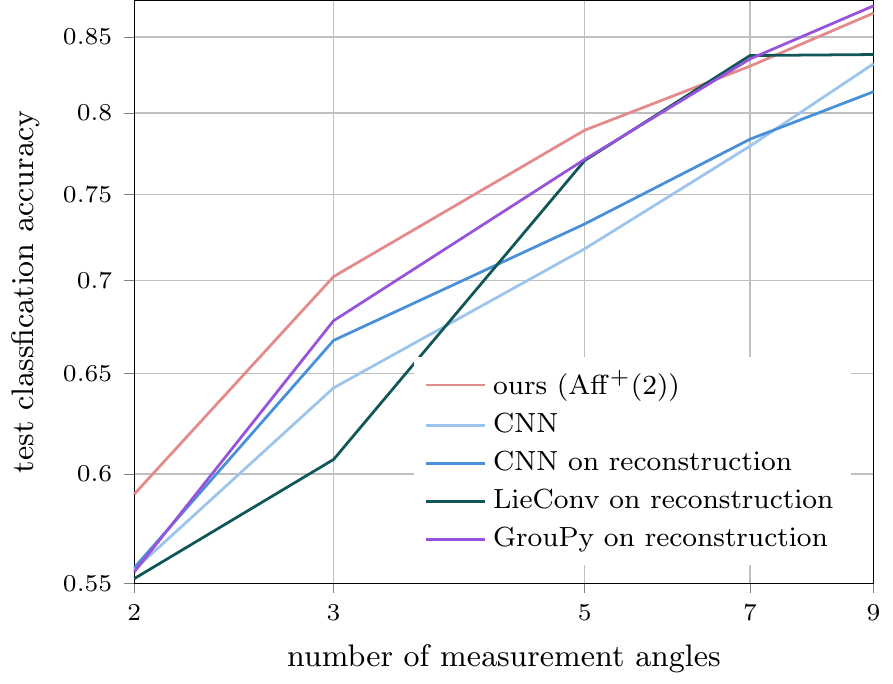}}
\caption{Classification accuracy for fan-beam CT projections of the RotMNIST dataset for different numbers of measurement angles when trained on RotMNIST and tested on RotMNIST or LinMNIST.}
\label{fig:classification}
\end{figure}

\subsection{Digit Classification}

As an exemplary task on indirect measurements, we perform experiments on the classification of sparse Radon measurements of digits from the popular MNIST~\cite{Deng2012} dataset and its variants, cf.~Figure~\ref{fig:radon_digit}.
We study invariance with respect to the roto-translation group $\SE$ and the group $\Aff$ of affine transforms with positive determinant we considered previously for the Radon operator in Example~\ref{ex:radon}.
To this end, we uniformly sample a rotation angle in $[0,2\pi)$, scaling factors for each direction in $[0.75, 1.25]$ and a horizontal shear factor in $[-0.5,0.5]$.
For the discretization of the integrals in the convolutions we always sample $2000$ points in the first residual block.

\subsubsection{Invariance}

To verify our implementation as well as discretization approach, we perform tests on the actual numerical invariance of the constructed networks.
Let us again stress that we can only expect {\em approximate} invariance due to the built-in approximations and sampling scheme.
For this purpose, we train roto-translation equivariant neural networks on varying geometries of sparse-angle fan-beam Radon measurements of the standard MNIST dataset (containing only upright digits) and evaluate the models on transformed versions of the test dataset, where we consider arbitrary rotations and transforms from $\GL$.
The transformed signals are computed using nearest-neighbour interpolation on the source images, before simulating the sparse measurements.
We train for $90$ epochs with a batch size of $14$.

\smallskip

The results for a varying amount of measurement angles are depicted in Figure~\ref{fig:invariance}, where the angles are chosen equidistantly in $[0, 2\pi)$, except for two angles, where we use $0$ and $\frac{\pi}{2}$.
While the classification accuracy on the non-rotated test data is similar for all geometries, the constructed invariance is better for increasing number of measurement angles.
Nevertheless, in all cases, we notice some extrapolation ability to unseen and transformed data, which underpins the intuition that the constructed network architecture serves as a good inductive bias even if exact equivariance cannot be guaranteed.
Surprisingly, the roto-translation equivariant networks perform similarly to the networks that were built to be equivariant with respect to the larger group $\Aff$ even in terms of generalization to unseen affinely transformed data.
This could indicate that the latter method requires additional fine-tuning.

\begin{figure}[t]
\centering
\subfigure[Samples from the LinMNIST dataset, which includes rotations, anisotropic scaling and shearing of MNIST digits.]{\includegraphics[width=0.4\textwidth]{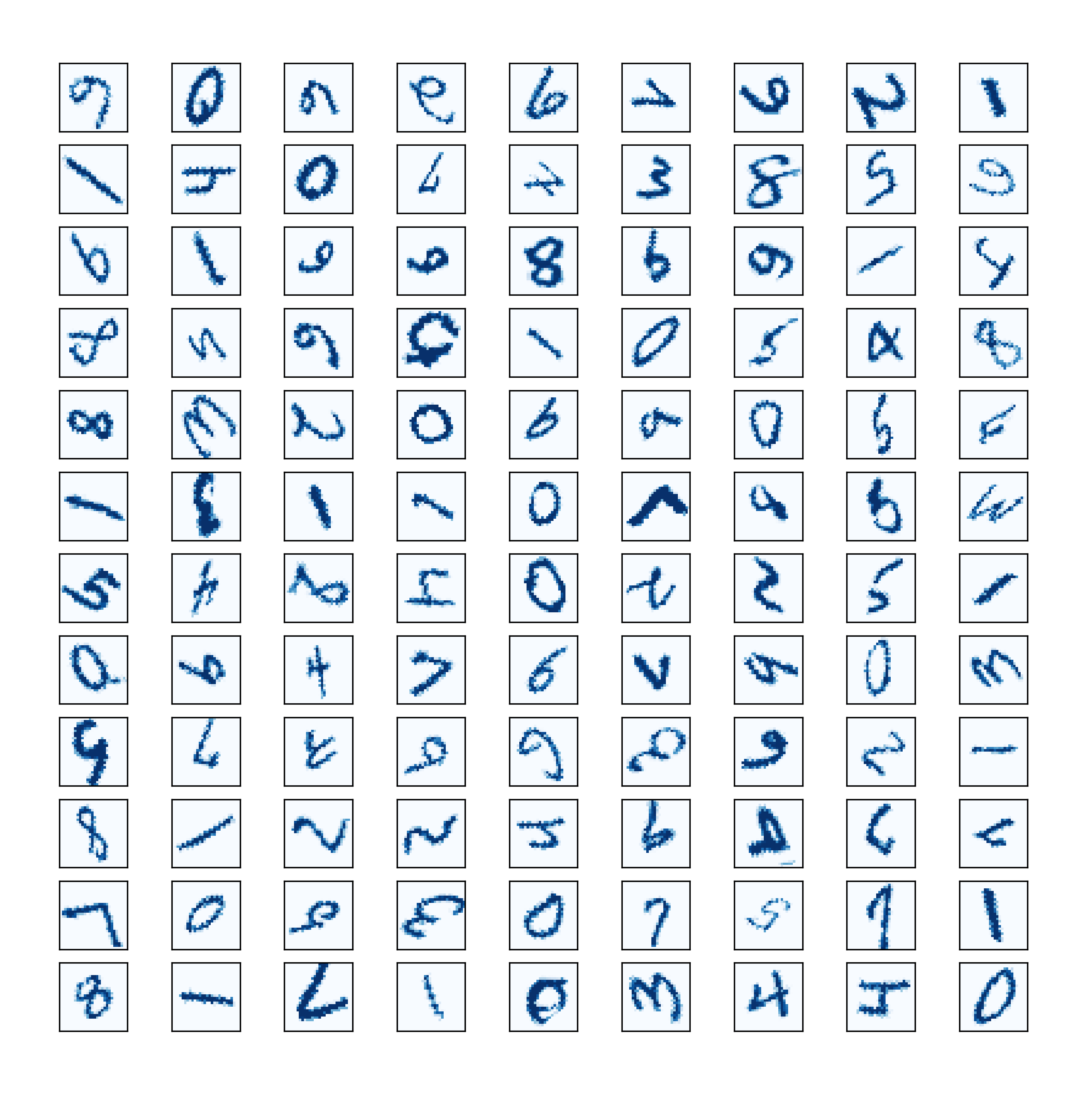}\label{fig:sample_digits}}
\hfil
\subfigure[Geometry of the simulated tube measurements, where $d_\text{max}$ and $d_\text{min}$ denote the maximal and minimal thickness of the tube.]{\includegraphics[width=0.4\textwidth]{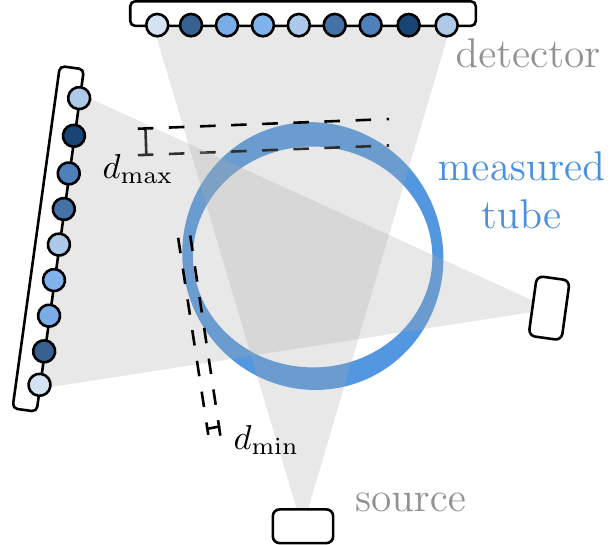}}
\caption{Visualizations of our generated datasets for the digit and ellipse experiments.}
\label{fig:datasets}
\end{figure}

\subsubsection{RotMNIST}

For further investigation we evaluate the classification accuracy of our approach on noisy indirect measurements of the rotated MNIST benchmark dataset, cf.~\cite{Larochelle2007}.
We consider the same measurement geometries and network architectures as before and add $6\%$ Gaussian noise to the simulated Radon data.
For comparison, we also train a standard discrete $\T$-equivariant CNN on the sparse sinograms and compute classical reconstructions for post-processing using the ASTRA toolbox implementation of the SIRT reconstruction scheme. An example for the obtained reconstructions in the 2-angle case is depicted in Figure~\ref{fig:radon_digit}.
On these reconstructions, we train a discrete group-equivariant (with respect to discrete translation and rotations in $\pi/2$-steps) ResNet18~\cite{He2015} based on the layer structure that was presented in~\cite{Cohen2016} as well as the continuous Lie-group-equivariant network implemented in LieConv~\cite{Finzi2020} for their experiments on RotMNIST. All models are trained for $400$ epochs with a batch size of $14$.
We evaluate the models on the RotMNIST test data as well as on a $\GL$-transformed version of it, which we refer to as LinMNIST. The results are visualized in Figure~\ref{fig:classification}, while samples from the LinMNIST dataset are depicted in Figure~\ref{fig:sample_digits}.

\smallskip

In all experiments, our approach outperforms the reconstruction-based approaches as well as the standard CNN, which does not encode proper symmetries in measurement space, in terms of generalization ability.
While the discrete reconstruction-based network shows better convergence on the training data, our method consistently obtains better results on the test dataset, especially for sparse geometries where the reconstructions are very inexact and amplify the noise.
These results show that we were able to construct a valuable prior for indirect measurements and that it can be advantageous to avoid reconstructions in certain cases.

\subsubsection{Limited-angle CT}

\begin{figure}[t]
\centering
\subfigure[Classification accuracy of $\SE$-equivariant networks for randomly rotated and $\GL$-transformed MNIST test measurements when exclusively trained on upright digits.]{\includegraphics[width=0.425\textwidth]{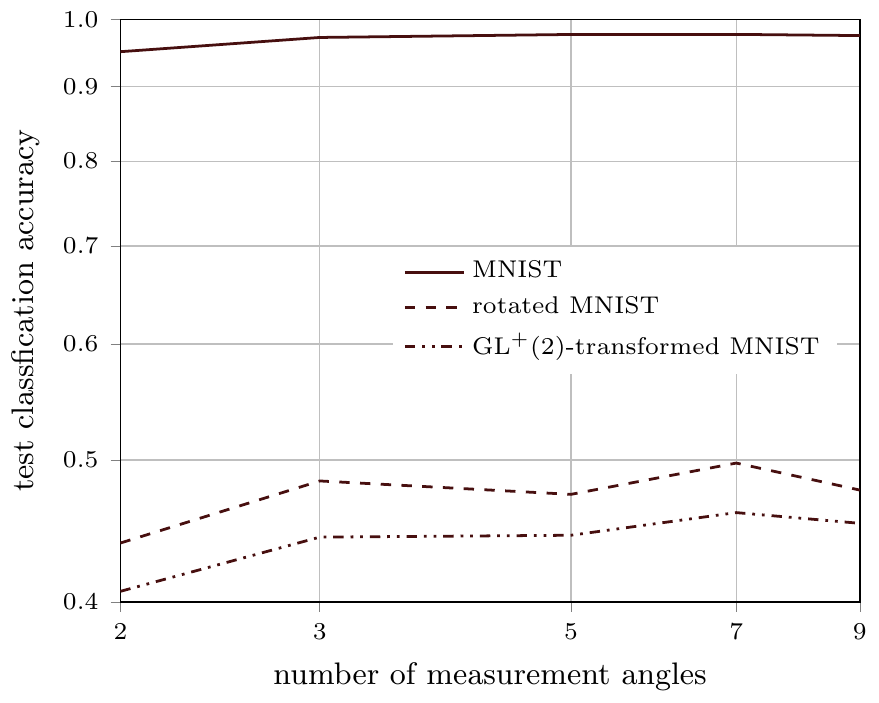}}
\hfil
\subfigure[Classification accuracy for RotMNIST test measurements for varying numbers of angles and different classification approaches.]{\includegraphics[width=0.425\textwidth]{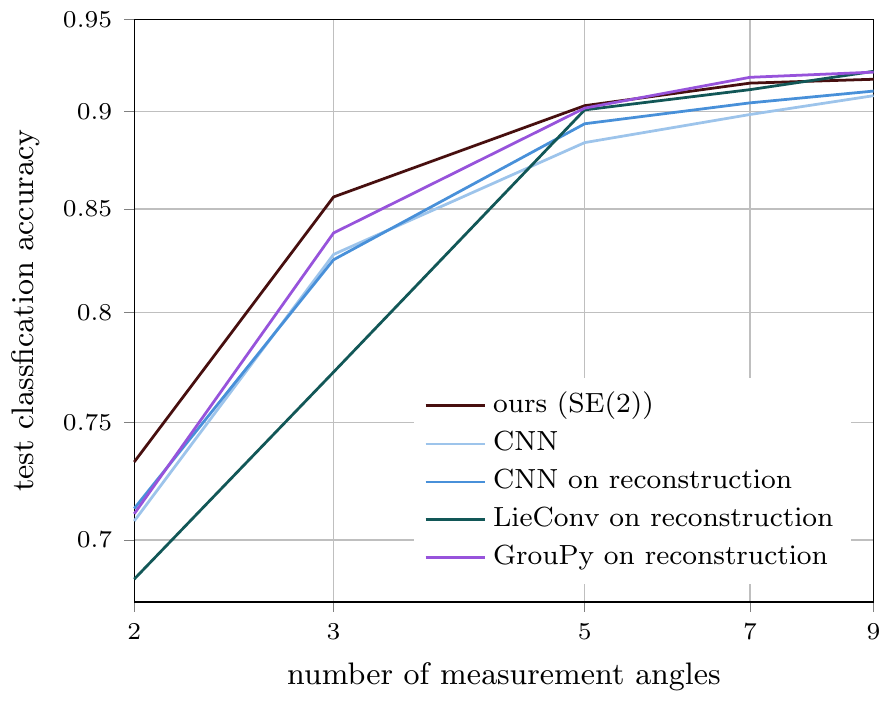}}
\caption{Invariance and classification accuracy for the case of limited-angle fan-beam CT geometries. While our network decreases in equivariance if the sensor locations do not cover the whole measurement space, it still provides an efficient bias when compared to other methods.}
\label{fig:lim_angle}
\end{figure}

To test the performance of our methods in applications where the sensor locations do not fully cover the measurement space, we additionally run experiments in the context of limited-angle CT. To this end, we also use a fan-beam geometry with a varying number of equidistant measurement angles, but choose them only in the range $[0,\frac{\pi}{2}]$. In this case, we do not expect good equivariance properties for any method due to the violation of the visibility condition even in the limit case of infinitely many data. If a digit has been measured only from a certain view angle during training, it becomes impossible to classify rotated version of it without assuming additional data knowledge.

\smallskip

The results using our proposed $\SE$-equivariant neural network architecture are depicted in Figure~\ref{fig:lim_angle}, where, in a first experiment, we train our roto-translation equivariant neural networks on limited-angle fan-beam Radon measurements of standard MNIST digits and evaluate the resulting models on rotated and $\GL$-transformed versions of the test dataset, and, in a second experiment, we train and evaluate various models on RotMNIST measurements.
In summary, we observe that especially for sparse angles the data efficiency of our proposed method still holds, while its equivariance cannot be guaranteed and does not improve with an increasing number of measurement angles in the limited range $[0,\frac{\pi}{2}]$.

\subsection{Thickness Regression of Tubes}

In addition to the above classification task, we also test a regression setting that is motivated by our initial example: the challenge of estimating the minimal and maximal thickness of a tube based on fan-beam data measured from only two angles, namely $0^\circ$ and $85^\circ$.
To simulate this task, we compute a dataset of elliptical rings, constructed from two ellipses with identical center point but randomly different rotation angles and radii.
We add $5\%$ Gaussian noise to the analytically computed sinogram, which is closer to real-world scenarios avoiding the simulation of forward measurements from given discrete data.
As ground truth, we approximate the minimal and maximal thickness of the tube by computing the length of the intersection with a line going through the center point of the ellipses.
For comparison, we compute reconstructed images of $129\times 129$ pixels using the SIRT scheme and compare to group-equivariant models that are trained on these reconstructions.
To evaluate the generalization performance, we evaluate different training dataset sizes, where all models are trained for $3000$ epochs on a batch size of $8$. In the initial residual block of our model, which is constructed to be equivariant to the induced representations of $\SE$, we sample $2700$ points for the estimation of the integrals.

Table~\ref{tab:ellipses_comparison} lists the obtained test errors.
We observe that, for dataset sizes below $8000$ training samples, our method outperforms both reconstruction-based approaches on the unseen data.
These results give further indication towards improved generalization properties and show the data-efficiency of our approach.

\begin{table}[t]
\footnotesize
\caption{Regression results on the generated ellipse dataset for varying amount of training data, where {\em MSE} refers to the mean squared error for both the minimal and maximal thickness. The methods 'GrouPy' and 'LieConv' act on previously computed classical reconstructions, while our methods directly handles the indirect measurement.}
\begin{center}
\begin{tabular}{@{}cccccccccc@{}}
\toprule
& \multicolumn{3}{c}{\textbf{MSE} $\times 10^{-3}$} & \multicolumn{3}{c}{\textbf{error at} $d_{\text{min}}$ $\times 10^{-2}$} & \multicolumn{3}{c}{\textbf{error at} $d_{\text{max}}$ $\times 10^{-2}$} \\   \cmidrule(lr{1em}){2-10}
dataset size & 
  \multicolumn{1}{c}{ours} &
  \multicolumn{1}{c}{GrouPy} &
  \multicolumn{1}{c}{LieConv} &
  \multicolumn{1}{c}{ours} &
  \multicolumn{1}{c}{Groupy} &
  \multicolumn{1}{c}{LieConv} &
  \multicolumn{1}{c}{ours} &
  \multicolumn{1}{c}{GrouPy} &
  \multicolumn{1}{c}{LieConv} \\ \midrule
1000 & \bf{0.93}      & 5.59      & 1.51     & \bf{1.95}      & 2.80       & 2.72     & \bf{2.49}      & 3.60       & 3.23     \\
2000 & \bf{0.81}      & 0.92      & 1.56     & \bf{1.80}       & 2.01      & 2.68     & \bf{2.35}      & 2.40       & 3.30      \\
4000 & \bf{0.70}      & 0.85      & 1.51     & \bf{1.67}      & 1.97      & 2.88     & \bf{2.18}      & 2.35      & 3.13     \\
8000 & 0.67      & \bf{0.51}      & 1.44     & 1.63      & \bf{1.50}       & 2.65     & 2.17      & \bf{1.82}      & 3.15     \\ \bottomrule
\end{tabular}
\end{center}
\label{tab:ellipses_comparison}
\end{table}

\section{Conclusion}

We have introduced neural networks that incorporate symmetries present in indirect measurements and theoretically investigated the connection between operators and symmetries.
Our key Theorem~\ref{thm:equivariance} precisely characterizes symmetries induced by given operators and, vice versa, determine all operators   that share the same symmetries.
Based on these insights we introduced data-efficient neural networks that are tailored to indirect measurements and can handle sparse inverse problems.
In particular, building upon our definition of generalized domain transforms, we were able to consider a class of symmetries that has not been treated before and shows new characteristics.
The effectivity of our approach is demonstrated on a classification and a regression task.

We want to stress that our work is meant to serve as a first attempt at the treatment of equivariance on indirectly measured data.
Several open questions and possible extensions arise both on the theoretical and the numerical side.
One direction could be the application and analysis of our approach in end-to-end reconstructions for sparse measurements.
In this setting, Theorem~\ref{thm:equivariance} indicates that a large part of the reconstruction operator is already encoded in the equivariance of our network.
Another important aspect is the incorporation of uncertainties and inexactness in the forward operator or its action on underlying symmetries in the source space. In this context, a Bayesian setting could help to also treat symmetries in noisy scenarios in a rigorous way, which is beyond the scope of this work.
Finally, an alternative to our approach of building equivariant network architectures on the whole measurement space~$\Y$ could be the construction of networks $\NN_\theta$ for discrete measurements in $\Yv$ such that the whole process $\NN_\theta\circ\S\circ\A: \X \to \Z$ is approximately equivariant.
We expect that this can be realized by training with augmented loss functions along with additional assumptions on the data distribution, which points to an interesting direction for future research.

%\clearpage
\vfill
%%%%%%%%%%%%%%%%%%%%%%%%%%%%%%%%%%

%\bibliographystyle{siamplain}
%\bibliography{references}

\end{document}